\documentclass[11pt]{amsart}

\usepackage{adjustbox,epsfig}

\usepackage{mathtools}
\usepackage{enumerate}
\usepackage{color}
 \usepackage[bookmarks=true,
bookmarksnumbered=true, breaklinks=true,
pdfstartview=FitH, hyperfigures=false,
plainpages=false, naturalnames=true,
colorlinks=false,
pdfpagelabels]{hyperref}

\makeatletter
\newtheorem*{rep@theorem}{\rep@title}
\newcommand{\newreptheorem}[2]{%
\newenvironment{rep#1}[1]{%
 \def\rep@title{#2 \ref{##1}}%
 \begin{rep@theorem}}%
 {\end{rep@theorem}}}
\makeatother
\newreptheorem{theorem}{Theorem}
\newreptheorem{corollary}{Corollary}

\usepackage{color}
\usepackage{verbatim}

\newcommand{\R}{\mathbb{R}} 

\newcommand{\Rn}{\mathbb{R}^n} 

\newcommand \B[1][n]{B_2^{#1}}
\newcommand{\N}{\mathbb{N}}
\newcommand{\vol}[1][n]{\operatorname{vol}_{#1}}
\newcommand{\Vol}{\operatorname{vol}_{nm}}
\def \s{\mathbb{S}^{n-1}}

\newcommand {\conbodo}[1][nm] {\mathcal{K}^{#1}_o}
\newcommand {\conbodio}[1][nm] {\mathcal{K}^{#1}_{(o)}}
\newcommand {\sta}[1][nm] {\mathcal{S}^{#1}}

\newcommand{\PP}[1][Q]{\Pi_{#1,p}^{\circ}}

\newcommand{\G}[1][Q]{\Gamma_{#1,p}}

\mathtoolsset{showonlyrefs}

\newtheorem{theorem}{Theorem}[section]
\newtheorem{definition}[theorem]{Definition}
\newtheorem{lemma}[theorem]{Lemma}
\newtheorem{proposition}[theorem]{Proposition}
\newtheorem{corollary}[theorem]{Corollary}

\title[Moment-Entropy inequalities for matrices]{On Moment-Entropy inequalities \\ in the space of matrices}

\author[Langharst]{Dylan Langharst}
\address{Institut de Math\'ematiques de Jussieu (IMJ-PRG) \\ Sorbonne Universit\'e, CNRS \\
4 Place Jussieu \\ 75252 Paris \\ France}
\email{dylan.langharst@imj-prg.fr}


\begin{document}

\begin{abstract}
    In a series of works, Lutwak, Yang and Zhang established what could be called \textit{affine information theory}, which is the study of moment-entropy and Fisher-information-type inequalities that are invariant with respect to affine transformations for random vectors. Their set of tools stemmed from sharp affine isoperimetric inequalities in the $L^p$ Brunn-Minkowski theory of convex geometry they had established. In this work, we generalize the affine information theory to the setting of matrices. These inequalities on the space of $n\times m$ matrices are induced by the interaction between $\R^n$ with its Euclidean structure and $\R^m$ equipped with a pseudo-norm.  
\end{abstract}

\maketitle

\section{Introduction}
A function $f$ on $\R^n$ is a density function if it is nonnegative, integrable, and integrates to $1$. We say $\mathcal{X}$ is a random vector with density $f$ if $\mathcal{X}$ is a function that has a finite $L^1$ norm with respect to the probability measure that has density $f$. To emphasize the connection between $\mathcal{X}$ and $f$, we will often write $f_{\mathcal{X}}$ for $f$. For $p\in \R$, we say that the random vector $\mathcal{X}$ has a finite moment of order $p$ if
\[
\int_{\R^n}|x|^pf_{\mathcal{X}}(x)dx < \infty.
\]
For $\lambda \in (0,\infty]$, the $\lambda$-R\'enyi entropy power $N_\lambda (\mathcal{X})$ of $\mathcal{X}$ is given by
\begin{equation}
\label{eq:renyi}
    N_\lambda (\mathcal{X}) = \begin{cases}
    \left(\int_{\R^n}f(x)^\lambda dx\right)^\frac{1}{n(1-\lambda)}, & \text{if }\lambda\neq 1,
    \\
    \exp{\left(-\frac{1}{n}\int_{\R^n}f(x)\log f(x) dx\right)}, & \text{if }\lambda = 1,
    \\
    \|f\|_{L^\infty(\R^n)}^{-\frac{1}{n}}, & \text{if }\lambda = \infty.
    \end{cases}
\end{equation}
Here, by $\|f\|_{L^\infty(\R^n)}$, we mean the essential supremum of $f$. We will be concerned with the interaction of matrices and density functions. If a random vector $\mathcal{X}$ has density $f_{\mathcal{X}}$ and $A$ is a non-singular matrix, then the density of the random vector $A\mathcal{X}$ has density
\[
f_{A\mathcal{X}}(y) = |\text{det} A|^{-1}f_{\mathcal{X}}(A^{-1}y).
\]
For a matrix $A,$ $A^t$ denotes its transpose, and $A^{-t}$ denotes its inverse transpose.

Consider two parameters, $\lambda,p>0$. Define the function on $\R_+$ given by
\begin{equation}
\label{eq:the_formula}
    p_{p,\lambda}(s)= \begin{cases}\left(1+(1-\lambda)\frac{s^p}{p}\right)^{-1 /(1-\lambda)}_+, & \text { if } \lambda \neq 1, \\ e^{-\frac{s^p}{p}}, & \text { if } \lambda=1.
    \end{cases}
\end{equation}
Here, for $c\in\R,c_+=\max\{c,0\}$. We note that when $\lambda< 1$, the function $p_{p,\lambda}$ is decreasing and nonnegative on all of $\R_+$, and so the $+$ subscript is not necessary in this range.

In 2004, Lutwak, Yang and Zhang \cite[Theorem 6.2]{LYZ04_3} showed the following moment-entropy inequality. Nguyen \cite{NVH19} later gave a different approach. For $p,\lambda>0,$ we say a random vector $\mathcal{Z}$ on $\R^n$ is a \textit{standard} generalized Gaussian random vector associated with $(p,\lambda)$ if its density is given by 
\begin{equation}
\label{eq:standard}
    f_{\mathcal{Z}}(x)= {c_{p,\lambda}}p_{p,\lambda}(|x|).
\end{equation}
Here, the normalizing constant, which ensures that $ f_{\mathcal{Z}}$ is a density function, is given by
\begin{equation}
    c_{p,\lambda} = \begin{cases}
        \omega_np^\frac{n}{p}\Gamma\left(1+\frac{n}{p}\right)\left(\frac{\Gamma\left(\frac{1}{1-\lambda}-\frac{n}{p}\right)}{(1-\lambda)^\frac{n}{p}\Gamma\left(\frac{1}{1-\lambda}\right)}\right), & \text{if }\lambda<1;
        \\
        \omega_n p^{\frac{n}{p}}\Gamma\left(1+\frac{n}{p}\right), & \text{if }\lambda =1;\\
        \omega_n p^\frac{n}{p}\Gamma\left(1+\frac{n}{p}\right)\left(\frac{\Gamma\left(\frac{\lambda}{\lambda-1}\right)}{(\lambda-1)^\frac{n}{p}\Gamma\left(\frac{\lambda}{\lambda-1}+\frac{n}{p}\right)}\right), & \text{if }\lambda >1.
    \end{cases}
\end{equation}
As usual, $|\cdot|$ is the Euclidean norm. We denote by $\Gamma(\cdot)$ the usual Gamma function. If a random vector has density $f_{A\mathcal{Z}}$ for some $A\in GL_n(\R)$, then we say $\mathcal{Z}$ is a generalized Gaussian random vector. Slight variations of the definition of generalized Gaussian random vector (with different coefficients adjacent to $|x|^p$) have appeared throughout \cite{LYZ04_3,LYZ05_2,LYZ07,LLYZ12,LLYZ13}. Density functions of the form \eqref{eq:standard} are also known as Barenblatt functions. Special cases have appeared in \cite{AS85,AE96,CT91,CI67,FKN90,JV07,AA09,KT09,MT10}.

We denote the dot product of two vectors $x$ and $y$ as $x\cdot y$. For $d\geq 0$, set 
$$\omega_d=\frac{\pi^\frac{d}{2}}{\Gamma\left(1+\frac{d}{2}\right)}.$$
We denote the unit Euclidean ball in $\R^n$ by $\B$; recall that $\vol(\B)=\omega_n$. We denote by $\s$ its boundary, the unit sphere.
\vskip 3mm
\noindent {\bf Theorem A}
\textit{Let $p\geq 1$ and $\lambda > \frac{n}{n+p}$. If $\mathcal{X}$ and $\mathcal{Y}$ are independent random vectors that have finite moment of order $p$, then it holds
\[
\mathbb{E}(|\mathcal{X}\cdot \mathcal{Y} |^p) \geq (n\omega_n)^{-\frac{n+p}{n}} \frac{2\omega_{n-2+p}}{(n\omega_n)^\frac{p}{n}\omega_{p-1}}  D_{n,p,\lambda}^2 [N_\lambda (\mathcal{X})N_\lambda (\mathcal{Y})]^p.
\]
The constant $D_{n,p,\lambda}$ is given by \eqref{eq_my_sharp} below.
There is equality if and only if there exists $A\in GL(n)$ and $a>0$ such that $X$ has density a.e. $f_{A\mathcal{Z}}$ and $\mathcal{Y}$ has density a.e. $g_{aA^{-t}\mathcal{Z}}$, where $\mathcal{Z}$ is a standard generalized Gaussian random vector associated with $(p,\lambda)$.
}
\vskip 3mm
\noindent Lutwak, Yang and Zhang had proven the inequality "from scratch" using tools from convex geometry. The use of convex geometry to establish inequalities in information theory was investigated early on in e.g. \cite{EL78,CC84,DCT91}. Nguyen, also using convex geometry, proved Theorem A starting with the following inequality. We denote by $\|f\|_{L^p(E)}$ the usual $L^p$ norm over a Borel set $E$ with respect to the Lebesgue measure. We note that if one compares the representation of Theorem A in the previously mentioned works to the one here, there will be a different exponent on $N_\lambda (\mathcal{X})$; this is due to a difference in definition. We follow here the definition of $N_\lambda (\mathcal{X})$ from the later work by Lutwak, Lv, Yang and Zhang \cite{LLYZ13}.
\vskip 2mm
\noindent {\bf Theorem B}
\textit{Let $p\geq 1$. Then, for any nonnegative functions $f,g\in L^1(\s),$ it holds that
$$\int_{\s}\int_{\s}|u\cdot v|^p f(u)g(v) dudv \geq  (n\omega_n)^{-\frac{n+p}{n}} \frac{2\omega_{n-2+p}}{(n\omega_n)^\frac{p}{n}\omega_{p-1}}  \|f\|_{L^{\frac{n}{n+p}}(\s)}\|g\|_{L^{\frac{n}{n+p}}(\s)}.$$
Equality holds if and only if there exists $A\in GL_n(\R)$ and constants $c_1,c_2>0$ such that, for a.e. $u\in\s$,
$$f(u) = c_1|A^{-1}u|^{-n-p} \quad \text{and} \quad g(u)=c_2 |A^{t}u|^{-n-p}.$$
}

Theorem B had been established by Lutwak and Zhang \cite[Theorem A]{LZ97} for continuous functions; Ngyugen then established it for all integrable functions by reducing to the continuous case. In this note, we establish a version of Theorem A in the setting of random vectors on the space of matrices. We denote by $M_{n,m}(\R)$ the space of $n\times m$ matrices. Like in the aforementioned works, we will be using concepts from convex geometry. Recall that a set $K\subset \R^d$ is said to be a convex body if, for every $x,y\in \R^d$, one has $(1-\lambda)x + \lambda y\in K$ for every $\lambda \in [0,1]$. In this work, we will also require our convex bodies to contain the origin. Then, we denote by $\conbodo[d]$ the set of convex bodies in $\R^d$. The support function of a convex body $K$ is given by $h_K(u)=\sup_{y\in K}\langle y,u \rangle$. We denote by $\vol[d](\cdot)$ the Lebesgue measure on $\R^d$; the Lebesgue measure on $M_{n,m}(\R)$ is then inherited by the identification of $M_{n,m}(\R)$ with $(\R^n)^m= \R^{nm}$. Elements in $M_{n,m}(\R)$ will be denoted with capital letters. Random vectors on $M_{n,m}(\R)$ will be denoted by $\mathfrak{X}$.

Let $\mathfrak{X}$ and $\mathcal{Y}$ be random vectors on $M_{n,m}(\R)$ and $\R^n$ that have density functions $f$ and $g$ respectively.  Then, for $Q\in \conbodo[m]$, we define the following expectation:
\begin{equation}
    \mathbb{E}[h_Q(\mathfrak{X}^t\mathcal{Y})^p]= \int_{\R^n}\int_{M_{n,m}(\R)}h_Q(A^tv)^pf_{\mathfrak{X}}(A)g_{\mathcal{Y}}(v)dAdv.
    \label{eq:expectation}
\end{equation}
If $Q$ is origin-symmetric (i.e. $Q=-Q$), then $h_Q$ is a norm. Other-wise, $h_Q$ is a pseudo-norm. The unit ball of this pseudo-norm is, by-definition, $Q^\circ$, the polar of $Q$. 
The idea behind the quantity \eqref{eq:expectation} is that the interaction between the random vector $\mathfrak{X}$ on $M_{n,m}(\R)$ and the random vector $\mathcal{Y}$ on $\R^n$ is occurring in the (pseudo)-Banach space $\R^m$ equipped with the pseudo-norm $h_Q$. In the space $M_{n,m}(\R)$, our analogue of $\B$ will not be $B_2^{nm}$, the unit ball with respect to the Frobenius norm. Instead, it will be a convex body $\PP \B\subset M_{n,m}(\R)$ which is generated from $\B$ and $Q$; we will discuss its origin in more detail below. Its definition is as follows. Given $K\in\conbodo[n],$ its gauge, which uniquely determines it, is precisely $\|x\|_K=\inf\{t>0:x\in tK\}$. Observe that $h_K = \|\cdot\|_{K^\circ}$. Then, $\PP \B$ is given by the gauge, for any matrix $U\in M_{n,m}(\R)$,
\[
\|U\|_{\PP \B} = \left(\int_{\s}h_Q(U^t\xi)^pd\xi\right)^\frac{1}{p}.
\]

For a fixed convex body $Q\in\conbodo[m]$, we say a random vector $\mathfrak{Z}$ on $M_{n,m}(\R)$ is a \textit{standard} generalized Gaussian random vector associated with $(Q,p,\lambda)$ if its density is given by
\begin{equation}
f_{\mathfrak{Z}}(U) =
        \alpha_{Q,p,\lambda}p_{p,\lambda}(\|U\|_{\PP \B}),
    \end{equation}
    where $\alpha_{Q,p,\lambda}>0$ is a normalizing constant. If a random vector has density $f_{A\mathfrak{Z}}$ for some $A\in GL_n(\R)$, then we say it is a generalized Gaussian random vector associated with $(Q,p,\lambda)$.

\begin{theorem}
\label{t:main}
    Fix $m,n\in\N$ and $Q\in\conbodo[m]$. Let $p\geq 1$ and $\lambda \geq \frac{nm}{nm+p}$. Suppose $\mathfrak{X}$ and $\mathcal{Y}$ are independent random vectors on $M_{n,m}(\R)$ and $\R^n$ respectively, each of which has a finite moment of order $p$. Suppose that either $\mathfrak{X}$ has even density or $Q$ is origin-symmetric. Then, 
    \begin{align*}\mathbb{E}[h_Q(\mathfrak{X}^t\mathcal{Y})^p\!] \geq \!D_{n,p,\lambda}D_{nm,p,\lambda}\! {(n\omega_n)^{-\frac{n+p}{n}}\left(nm\Vol(\!\PP \B\! )\right)^{-\frac{p}{nm}}}[N_\lambda(\mathfrak{X})N_\lambda (\mathcal{Y})]^p.\end{align*}
    Here, $D_{n,p,\lambda}$ is a sharp constant given by \eqref{eq_my_sharp} below.
    There is equality if and only if there exists $A\in GL_n(\R)$ and $\alpha>0$ such that $\mathcal{Y}$ has density $g_{\alpha A^{-t}\mathcal{Z}}$ where $\mathcal{Z}$ is a standard generalized Gaussian random vector associated with $(p,\lambda)$ and $\mathfrak{X}$ has density $g_{A\mathfrak{Z}}$ where $\mathfrak{Z}$ is a standard generalized Gaussian random vector associated with $(Q,p,\lambda)$.
\end{theorem}

We follow the route suggested by Ngyugen, by first establishing the appropriate analogue of Theorem B. 
\begin{theorem}
\label{t:BS_on_sphere}
    Fix $n,m\in\N,p\geq 1$ and $Q\in\conbodo[m]$. Let $f\in L^1(\mathbb{S}^{nm-1})$ and $g\in L^1(\s)$. If either $f$ is even or $Q$ is symmetric, then one has
    \begin{align*}&\int_{\s}\int_{\mathbb{S}^{nm-1}}h_Q(U^tv)^p f(U)g(v)dUdv 
    \\
    &\geq \|f\|_{L^\frac{nm}{nm+p}(\mathbb{S}^{nm-1})}\|g\|_{L^\frac{n}{n+p}\left(\s\right)} (n\omega_n)^{-\frac{n+p}{n}}\left(nm\Vol(\PP \B )\right)^{-\frac{p}{nm}}.
    \end{align*}
There is equality if and only if there exists a matrix $A\in GL_n(\R)$ and nonnegative constants $c_1,c_2 >0$ such that $g(u)=c_1|A^{t}u|^{-(n+p)}$ and $f(U) =c_2\|A^{-1}U\|_{\PP \B}^{-(nm+p)}$ almost everywhere.
\end{theorem}

In Section~\ref{sec:preliminaries}, we recall some recent concepts from convex geometry. In Section~\ref{sec:proofs}, we prove our main results. We will also extend some other moment-entropy inequalities, but we save the discussion of these topics for Section~\ref{sec:other_results}. For the reader interested in other works that involve entropy inequalities using convex geometry, we recommend the survey \cite{MMX17} and e.g. the more recent works \cite{ENT18,MM19,FLM20,EG24,Met21,MNT21}. In particular, the works \cite{HLXY22} and \cite{HJM20} are relevant. The use of Gaussian probability in convex geometry, as the object of interest, can be found in e.g. \cite{EHR1,EHR1.5,EHR2,KS93,Ball93,Lat96,CEFM04,BS07,Zv08,LMNZ17,GZ10,KL21,EM21,HXZ21,JL22,FHX23,FLX23,FLMZ23_1,FLMZ23_2}. The generalized Gaussian density was studied in \cite{LS24,LXS24}.

{\bf Acknowledgments:} We would like to thank Autt Manui for the valuable suggestions given on a draft of the present manuscript.

{\bf Funding:} The author was funded by a post-doctoral fellowship provided by the Fondation Sciences Math\'ematiques de Paris.

\section{Preliminaries}
\label{sec:preliminaries}
Let $K\in\conbodo[n]$. Then, one has the following set of inequalities: 
\begin{equation}
\label{eq:RS}
2^n \leq \frac{\vol(DK)}{\vol(K)} \leq \binom{2n}{n},\end{equation}
where $DK=\{x\in\R^n: K\cap(K+x)\neq \emptyset\}$ is the difference body of $K$, the lower-bound follows from the Brunn-Minkowski inequality, with equality if and only if $K$ is symmetric, and the upper-bound is the Rogers-Shephard inequality, with equality if and only if $K$ is an $n$-dimensional simplex \cite{RS57}. We note that, alternatively, one can write $DK=K+(-K)$, where $A+B=\{a+b:a\in A,b\in B\}$ is the Minkowski sum of Borel subsets $A$ and $B$.

For $m\geq 1$, Schneider introduced \cite{Sch70} an $m$th-order analogue of this inequality. Firstly, he defined the $m$th-order difference body $D^m(K)\subset \R^{nm}$ as 
$$D^m(K) = \left\{(x_1,\dots,x_m)\in(\R^n)^m: K\bigcap_{i=1}^m (K+x_i)\neq \emptyset\right\}.$$
Then, he showed the following generalization of the Rogers-Shephard inequality:
$$\vol[nm](D^m(K))\vol[n](K)^{-m} \leq \binom{nm+n}{n},$$
again with equality if and only if $K$ is an $n$-dimensional simplex. Concerning a lower-bound, if $m=1$ and $n\in\N$, or $n=2$ and $m\in\N$, then the lower-bound is obtained by all origin symmetric convex bodies. If $n\geq 3$ and $m\geq 2$, this is not the case; Schneider's conjecture is precisely that the lower-bound is obtained by ellipsoids for $n$ and $m$ in this range. The importance of Schneider's conjecture, in addition to its difficulty, was shown in \cite{Sch20}; it is equivalent to Petty's conjecture for the volume of the projection body when $n=3$.

It was shown in \cite{HLPRY23,HLPRY23_2,LPRY23,LX24,LMU24,JH24,LRZ24,YZZ25}, that the $m$th-order concept has led to the burgeoning $m$th-order Brunn-Minkowski theory. Mainly, the $m$th-order concept is extended to other functionals from the $L^p$ Brunn-Minkowski theory; the $L^p$ Brunn-Minkowski theory has its origins in the work by Firey \cite{Firey62}, Lutwak \cite{LE93,LE96}, Lutwak and Zhang \cite{LZ97} and Lutwak, Yang and Zhang \cite{LYZ00,LYZ02,LYZ04,LYZ04_2,LYZ04_3,LYZ05,LYZ05_2}.  (This list is by no means exhaustive!) 

There are two inequalities in particular that will be pertinent. We denote by $\conbodio[d]$ the set of convex bodies in $\R^d$ that contain the origin in their interiors. We will take $d\in \{n,nm\}$. For $p\geq 1,$ the $L^p$ projection body of $K$ is the origin-symmetric convex body $\Pi^\circ_p K\in \conbodio[n]$ whose gauge, which is a norm, is given by
\[
\|\theta\|_{\Pi^\circ_p K}= \left(\int_{\s}\left(\frac{|u\cdot \theta|}{h_K(u)}\right)^ph_K(u)d\sigma_K(u)\right)^\frac{1}{p}.
\]
Here, $\sigma_K$ is the surface area measure of $K\in\conbodo[n]$ on $\s$, defined as, for $E\subset\s$ Borel,
\begin{equation}\sigma_K(E)=\int_{n^{-1}_K(E)}d\mathcal{H}^{n-1}(y),
\label{eq:surface_area}
\end{equation}
where $\mathcal{H}^{n-1}$ is the $(n-1)$-dimensional Hausdorff measure and $n_K$ is the Gauss map of $K$ which pushes $\partial K$, the boundary of $K$, to $\s$. The formula of the volume of a convex body $K$ is then given by 
\begin{equation}
    \vol[n](K)=\frac{1}{n}\int_{\s}h_K(u)d\sigma_K(u).
    \label{eq:convex_vol}
\end{equation}

 The following inequality, the so-called $L^p$ Petty's projection inequality, was shown in \cite{LYZ00}, extending the $p=1$ case by Petty \cite{CMP71}.

\noindent {\bf Theorem C}
\textit{
Let $p\geq 1$. Then, for every $K\in\conbodio[n]$, we have
\[
\vol[n](K)^{\frac{n-p}{p}}\vol[n](\Pi^\circ_p K) \leq \vol[n](\B)^{\frac{n-p}{p}}\vol[n](\Pi^\circ_p \B),
\]
with equality if and only if $K$ is a linear (or affine if $p=1$) image of $\B$.
}
\vskip 3mm
\noindent The body $\Pi^\circ_p \B$ is a dilate of $\B$ whose radius, under the normalization used presently, is $\left(\frac{\omega_{p-1}}{2\omega_{n+p-2}}\right)^\frac{1}{p}$. An important application of Theorem C is the $L^p$ Affine Sobolev inequality from \cite{GZ99,LYZ00,TW12}: suppose $w \in W^{1,p}(\R^n)$, that is $w$ has a weak derivative $\nabla w$ that is in $L^p(\R^n).$ Then, if $1 < p < n$,  one has 
 \begin{equation} \label{eq:LYZAffineSobolev}
 \begin{split}
\left(\frac{n\omega_n\omega_{p-1}}{2\omega_{n+p-2}}\right)^\frac{1}{p}\!\left(\int_{\s}\! \left(\int_{\Rn}|\langle \nabla w(v),\theta\rangle|^p dv \right)^{-\frac{n}{p}}\!\frac{d\theta}{n\omega_n}\right)^{-\frac{1}{n}}\! \geq \!a_{n,p}\|w\|_{L^\frac{np}{n-p}(\R^n)}.
\end{split}
 \end{equation}
Here, $a_{n,p}$ is the sharp constant from the Aubin-Talenti $L^p$ Sobolev inequality \cite{Aubin1,Talenti1}: one has $a_{n,1}=n\omega_n^\frac{1}{n}$, and, for $p>1$,
\begin{equation}
    a_{n,p} = n^\frac{1}{p}\left(\frac{n-p}{p-1}\right)^\frac{p-1}{p}\left(\frac{\omega_n}{\Gamma(n)}\Gamma\left(\frac{n}{p}\right)\Gamma\left(n+1-\frac{n}{p}\right)\right)^\frac{1}{n},
    \label{eq:sobolev_cons}
\end{equation}
In fact, \eqref{eq:LYZAffineSobolev} is sharper than the $L^p$ Sobolev inequality. There is an equality in \eqref{eq:LYZAffineSobolev} when $w$ is of the form $w(v) = \left(\alpha + |A(v-v_0)|^{\frac{p}{p-1}}\right)^{-\frac{n-p}{p}},$ with $A$ a non-singular $n\times n$ matrix, $\alpha >0$, and $v_0 \in \Rn$. If $p=1$, then the inequality can be extended to functions of bounded variation, in which case equality holds if and only if $w$ is a multiple of $\chi_E$ for some $n$-dimensional ellipsoid $E$. Here, the characteristic function of $E$ is given by $\chi_E(x)=1$ if $x\in E$ and $0$ otherwise.

In \cite{HLPRY23_2}, the following extension of the body $\Pi^\circ_p K$ and the $L^p$ Petty's projection inequality to the setting of matrices was established.
\begin{definition}
\label{def:hi_gauge}
 For $K\in\conbodio[n]$, $Q\in\conbodo[m]$ and $p\geq 1$, its $(L^p,Q)$ polar projection body $\PP K\in\conbodio[nm]$ is given by the gauge $$\|\Theta\|_{\PP K } = \left( \int_{\s} \left(\frac{h_Q(\Theta^t\xi)}{h_K(\xi)}\right)^p h_K(\xi)d\sigma_K (\xi)\right)^\frac 1p .$$   
\end{definition}

 We note that the operator $\PP$ goes from $\conbodio[n]$ to $\conbodio[nm]$. The connection between the work of Schneider and $\PP$ is a bit tenuous, but it was shown in \cite{HLPRY23} that they are related when $p=1$ and $Q$ is the orthogonal simplex of dimension $m$. The associated isoperimetric inequality for $\PP$, the $(L^p,Q)$ Petty's projection inequality, was shown in \cite[Theorem 1.3]{HLPRY23_2}.
 \vskip 2mm
 \noindent {\bf Theorem D} 
 \textit{Let $m,n \in \N$ and $p \geq 1$. Then, for any pair of convex bodies $K \in \conbodio[n]$ and $Q \in \conbodo[m]$, one has 
\begin{equation}\label{e:PPIGeneral}
\Vol(\PP K ) \vol(K)^{\frac{nm} p -m} \leq \Vol(\PP \B ) \vol(\B)^{\frac{nm} p -m}.
\end{equation}
There is equality if and only if $K$ is a linear (or affine if $p=1$) image of $\B$.
}
\vskip 3mm
The aforementioned $L^p$ Petty's projection inequality is recovered by setting $m=1$ and $Q=[-1,1]$. But also, setting $Q$ to be any interval in $\R$ containing the origin, Theorem D recovers the asymmetric $L^p$ Petty's projection inequality by Haberl and Schuster \cite{HS09}. An application of Theorem D is the $(L^p,Q)$ Sobolev inequality \cite[Theorem 1.1]{HLPRY23_2}, which contains both \eqref{eq:LYZAffineSobolev} and the asymmetric $L^p$ Sobolev inequality by Haberl and Schuster \cite{HS2009} in a larger framework.

We mention two key tools that can be used to prove \eqref{eq:LYZAffineSobolev}, as this approach will prove fruitful to our current considerations. The first is the $LYZ$ body. Given a function $w \in W^{1,p}(\R^n)$ which is non-constant, Lutwak, Yang and Zhang showed \cite{LYZ06} the existence of its $p$th $LYZ$-body: it is the unique, origin-symmetric convex body $\langle w\rangle_p\in\conbodio[n]$, satisfying the following change of variables formula 
\begin{equation}\int_{\s}\left(\frac{g(u)}{h_{\langle w \rangle_p}(u)}\right)^ph_{\langle w \rangle_p}(u)d\sigma_{\langle w \rangle_p}(u)=\int_{\R^n}g(-\nabla w(v))^pdv
\label{eq:LYZw}
\end{equation}
for every $1$-homogeneous, nonnegative function \textit{even} $g$ on $\R^n$. The anisotropic-Sobolev inequality from \cite{MS86,CENV04} implies that
\begin{equation}
\label{eq:bound_LYZ}
\vol(\langle w\rangle_p)^{\frac{n-p}{np}}\geq \frac{a_{n,p}}{n^\frac{1}{p}\omega_n^{\frac{1}{n}}}\|w\|_{L^\frac{np}{n-p}(\R^n)}.\end{equation}
Setting in \eqref{eq:LYZw} $g(u)=|u\cdot \theta|$ for a fixed $\theta\in\s$, applying Theorem C to the body $\langle w \rangle_p$ and then \eqref{eq:bound_LYZ} yields the $L^p$ affine Sobolev inequality \eqref{eq:LYZAffineSobolev}.

As for the Sobolev inequality from \cite{HLPRY23_2}  (and \cite{HS2009} for the case where $Q$ is an asymmetric interval), $h_Q$ will only be even if $Q$ is origin-symmetric. Thus, a few technical details were introduced to overcome this hurdle. Later in Section~\ref{sec:other_results}, we will consider only origin-symmetric $Q$, putting our situation closer to the framework by Lutwak, Yang and Zhang. The above argument we sketched out will be helpful for our present considerations.


First, recall the sharp Gagliardo-Nirenberg inequality established by Del Pino and Dolbeault \cite{DD02} in the case of the Euclidean ball and by Cordero-Erausquin, Nazaret, and Villani \cite{CENV04} in general: let $K$ be an origin-symmetric convex body such that $\vol[n](K)=\omega_n$. Then, if $1\leq p <n$ and $w\in W^{1,p}(\R^n)$ is a non-constant function there exists a sharp constant $c(n,r,p)$ such that, for every $0<r\leq \frac{np}{n-p}$ it holds
\begin{equation}
    \label{eq:CNV}
    \left(\int_{\R^n}h_K(\nabla w(x))^pdx\right)^\frac{1}{p} \geq c(n,r,p)\|w\|_{L^q(\R^n)}^{1-s}\|w\|_{L^r(\R^n)}^s,
\end{equation}
where 
\begin{equation}
\label{eq:q_param}
q=1+r\frac{p-1}{p},
\end{equation}
and $s$ is for scale invariance, i.e. 
\begin{equation}\frac{n-p}{np}=\frac{1-s}{q}+\frac{s}{r}.
\label{eq:s_param}
\end{equation} The constraining on the size of $K$ ensures the constant $c(n,r,p)$ is independent of $K$.
There is equality only for functions of the form 
\begin{equation}
\label{CNV_equality}
w_{a,b,K}(x):= bp_{\frac{p}{p-1},2-\frac{r}{p}}\left(a\|x-x_0\|_{K}\right)
\end{equation}
for some $a,b>0$ and $x_0\in\R^n$. The function $p_{p,\lambda}$ was given in \eqref{eq:the_formula}; notice the cases $\lambda \neq 1$ and $\lambda=1$ become $p \neq r$ and $p = r$.

Consider the case of \eqref{eq:CNV} when $K=\left(\frac{\omega_n}{\vol[n](\langle w \rangle_p)}\right)^{\frac{1}{n}}\langle w \rangle_p$. Then, \eqref{eq:LYZw} and \eqref{eq:convex_vol} yield
\begin{equation}
    \label{eq:CNV_2}
    \vol[n](\langle w \rangle_p)^\frac{n-p}{np} \geq n^{-\frac{1}{p}}\omega_n^{-\frac{1}{n}}c(n,r,p)\|w\|_{L^q(\R^n)}^{1-s}\|w\|_{L^r(\R^n)}^s,
\end{equation}
with equality if and only if $w$ has the form \eqref{CNV_equality} for some $K$. This result is precisely \cite[Theorem 6.1]{LYZ06}; in fact, this is their proof, which we have included here for completeness because it is short. We need \eqref{eq:CNV_2} for this next step: using \eqref{eq:CNV_2} and Theorem D for an origin-symmetric $Q\in\conbodio[m]$, we obtain for $1\leq p <n$ that
\begin{equation}
    \label{eq:CNV_3}
    \vol[nm](\PP \langle w \rangle_p)^{-\frac{1}{mn}} \geq \Vol(\PP \B )^{-\frac{1}{nm}} (n\omega_n)^{-\frac{1}{p}} c(n,r,p)\|w\|_{L^q(\R^n)}^{1-s}\|w\|_{L^r(\R^n)}^s.
\end{equation}
There is equality if and only if $w$ has the form \eqref{CNV_equality} with $K=\B$ (one can verify that $\langle w_{a,b,\B} \rangle_p$ is a Euclidean ball). The inequality \eqref{eq:CNV_3} can be seen as an extension of \cite[Theorem 7.2]{LYZ06}. We mention other Sobolev-type inequalities can be found in \cite{DD02,CEGH04,CENV04,CER23}.

We now discuss the second inequality from convex geometry that is relevant to our considerations. The Blaschke-Santal\'o inequality states that if $K$ is a convex body such that $K$ or $K^\circ$ has center of mass at the origin, then it holds
\begin{equation}
\label{eq:BS}
\vol(K)\vol(K^\circ) \leq \vol(\B)^2,
\end{equation}
with equality if and only if $K$ is a linear image of $\B$. Please, see the survey \cite{FMZ23} for the history of this inequality. 

Lutwak and Zhang \cite{LZ97} extended \eqref{eq:BS} to both the $L^p$ setting and the setting of star bodies. Recall that $L\subset \R^d$ is a star body if $L$ is compact, $o\in L,$ $x\in L$ implies the segment $[o,x]\subset L$, and $x\mapsto \|x\|_L^{-1}$ is continuous on $\R^d\setminus\{o\}$. Every convex body is a star body. The radial function of a star body $L$ is precisely $\rho_L=\|\cdot\|_L^{-1}$. That is
\[
\rho_L(u)=\sup\{t>0:tu\in L\}.
\]
We denote by $\mathcal{S}^d$ the set of star bodies in $\R^d$. The volume of a star body is given by
\begin{equation}
    \label{eq:volume_star}
    \vol[d](L) = \frac{1}{d}\int_{\mathbb{S}^{d-1}}\rho_L(u)^d du.
\end{equation}

For $p\geq 1$, Lutwak and Zhang defined the $L^p$ centroid body $\Gamma_p L\in \conbodio[n]$, which is origin-symmetric, of a star body $L$ via the support function
$$h_{\Gamma_p L}(\theta) = \left(\frac{1}{\vol[n](L)}\int_L |\theta\cdot x|^pdx\right)^\frac{1}{p}.$$
Notice that $(\Gamma_p L)^\circ \to L^\circ \cap (-L)^\circ$ when $p\to \infty$. The main result of Lutwak and Zhang was to generalize \eqref{eq:BS} by replacing $K$ with a star body $L$ and replacing $K^\circ$ with $(\Gamma_p L)^\circ$. 

We now recall the recently shown analogue of this fact in the setting of matrices.
\begin{definition}
  \label{d:generalcentroidbody}
  Let $p \geq 1$, $m \in \N$, and fix some $Q \in \conbodo[m]$. Given a compact set $L\subset M_{n,m}(\R)$ with positive volume, the $(L^p, Q)$-centroid body of $L$, $\G L$, is the convex body in $\R^n$ with the support function 
  \begin{equation}
  h_{\G L}(v)= \left(\frac 1 {\Vol(L)}\int_L h_Q(A^tv)^p dA\right)^\frac{1}{p}.
  \label{eq:cen_hi}
\end{equation}
\end{definition}

With this definition in hand, we have the following extension of \eqref{eq:BS} from \cite{HLPRY23_2}. In addition to the $Q=[-1,1]$ cases studied in \cite{LZ97,LYZ00}, Haberl and Schuster analyzed the case $m=1$ in its entirety \cite{HS09}; in our language, their work is when $Q=[-\alpha_1,\alpha_2],\alpha_i\geq 0$. We highlight that $\G$ takes a compact set in $M_{n,m}(\R)$ and produces a convex body in $\R^n$. We use the natural notation $\G^\circ L=(\G L)^\circ$.
\begin{lemma}[The $m$th-order $L^p$ Santal\'o inequality]
\hfill\break
\label{l:hipBS}
  Fix $p\geq 1$ and $Q\in\conbodo[m].$ Consider a compact set with non-empty interior $L\subset M_{n,m}(\R)$ with positive volume. Then, if $\G L$ or $\G^\circ L$ has center of mass at the origin,
  \begin{equation}
  \begin{split}\Vol(L)^{\frac 1 m }\vol(\G^\circ L)\leq \vol(\B)^2 \frac{\Vol(\PP \B )^{\frac 1 m }}{\vol(\G \PP \B )},
  \end{split}
\label{eq:BS_2}
\end{equation}
with equality throughout if and only if $L = \PP E $ for some origin-symmetric ellipsoid $E \in \conbodio[n]$.
\end{lemma}
The requirement that $\G L$ or $\G^\circ L$ has center of mass at the origin in Lemma~\ref{l:hipBS} is automatically satisfied when $Q$ or $L$ is symmetric. When $E$ is a centered ellipsoid, it was shown that 
\begin{equation}\G \PP E = \vol(\B)^{\frac{1}{p}}\vol(E)^{-\frac 1 {p}} \left(\frac m {\vol(\B)(nm+p)}\right)^{\frac 1p } E.
\label{eq:elipp_cal}
\end{equation}
  This yields $\lim_{p\to\infty}\G\PP E=E$. 

  In the next section, we show how the $m$th-order $L^p$ Santal\'o inequality in Lemma~\ref{l:hipBS} yields our main result, Theorem~\ref{t:main}. We need a few more fundamental facts. Given a measurable function, we can associate a star body to it. If $f$ is a nonnegative, measurable function on $\R^n$, then its $p$th Ball body, for $p>0$, is the star-shaped set given by
\begin{equation}
    K_f(p) = \left\{x\in\R^n:\int_0^\infty f(rx)r^{p-1}dr\geq 1\right\},
\end{equation}
and thus the radial function of $K_f(p)$ is given by
\begin{equation}
    \label{eq:ball_body}
    \rho_{K_f(p)}(\theta) = \left(\int_0^\infty f(r\theta)r^{p-1}dr\right)^\frac{1}{p}.
\end{equation}
If the integral is finite for every $v\in\s$, then $K_f(p)$ is a star-body. In \cite{Ball88,GZ98} it was shown that if $f$ is log-concave, then $K_f(p)$ is a convex body. A nonnegative function $f$ being log-concave is precisely that, for every $x,y$ such that $f(x)f(y)>0$ and $\lambda \in [0,1],$ one has
\[
f((1-\lambda)x+\lambda y)\geq f(x)^{1-\lambda}f(y)^{\lambda}.
\]

  We next loosen the definition of star bodies. For $s\in\R$, we say $L\subset\R^d$ is an $(s)$-star, and write $L\in \sta[d]_s$ if it is star-shaped and its radial functional is merely in $L^{s}(\s)$. Notice that $\sta[d]\subset \cup_{s\in \R}\sta[d]_s$. For $p\geq 1$, the dual mixed volume of $K\in \sta[d]_{d+p}$ and $L\in \sta[d]_{-p}$ is given by 
\begin{equation}
  \widetilde V _{-p,d}(K,L)=\frac 1d \int_{\mathbb{S}^{d-1}}\rho_K (\theta)^{d+p}\rho_L(\theta)^{-p} d\theta.
  \label{eq:dual_mixed}
\end{equation}
With the same assumptions on $K,L$ and $p,$ one has the dual Minkowski's first inequality:
\begin{equation}
 \widetilde V _{-p,d}(K,L)^d \geq \vol[d](K)^{d+p}\vol[d](L)^{-p},
  \label{dual_Min_first}
\end{equation}
with equality if and only if $K$ and $L$ are dilates. When $K$ and $L$ are star bodies, this was shown by Lutwak \cite{Lut75} using H\"older's inequality; the generality presented here was shown in \cite{LYZ04_3}.

\section{Proof of our main results}
\label{sec:proofs}
To prove our main results, we start with the following inequality.
\begin{lemma}
\label{l:LZBS}
    Fix $n,m\in\N$ and $p\geq 1$. Let $Q\in\conbodo[n],K\in\sta[n]_{n+p}$ and let $L\subset  M_{n,m}(\R)$ be a compact set with non-empty interior. Then, if $\G L$ or $\G^\circ L$ has center of mass at the origin, one has
     \begin{align*}
        \int_K\int_L&h_Q(A^tu)^p dA du 
        \\
        &\geq \frac{nm}{(n+p)(nm+p)}\frac{\vol[nm](L)\vol[n](K)}{\vol(\B)}\left(\frac{\vol[nm](L)}{\Vol(\PP \B )}\left(\frac{\vol[n](K)}{\vol[n](\B)}\right)^m\right)^\frac{p}{nm},
    \end{align*}
    with equality if and only if one has up to sets of measure zero that $K$ is a dilate of $\G^\circ L$ and $L=\PP E$ for some centered ellipsoid $E\in\conbodio[n]$.
\end{lemma}

\begin{proof}
    Observe that we have
    \begin{align*}
        &\int_K\int_Lh_Q(A^tu)^p dA du = \int_{\s}\int_{L}h_Q(A^tv)^pdA\left(\int_{0}^{\rho_K(v)}r^{n+p-1}dr\right)dv
        \\
        &=\frac{1}{n+p}\int_{\s}\int_{L}h_Q(A^tv)^pdA\rho_K(v)^{n+p}dv
        \\
        &=\frac{\vol[nm](L)}{n+p}\int_{\s}\left(\frac{1}{\vol[nm](L)}\int_{L}h_Q(A^tv)^pdA\right)\rho_K(v)^{n+p}dv
         \\
        &=\frac{\vol[nm](L)}{n+p}\int_{\s}\rho_{\G^\circ L}(v)^{-p}\rho_K(v)^{n+p}dv,
    \end{align*}
    where, in the final line, we used that $\rho_{\G^\circ L}(v) = h_{\G L}(v)^{-1}$ and inserted \eqref{eq:cen_hi}. Then, from \eqref{eq:dual_mixed} and \eqref{dual_Min_first}, we obtain
    \begin{align*}
        \int_K\int_Lh_Q(A^tu)^p dA du &=\frac{\vol[nm](L)}{n+p}\int_{\s}\rho_{\G^\circ L}(v)^{-p}\rho_K(v)^{n+p}dv
        \\
        &=\frac{n}{n+p}\vol[nm](L)\widetilde V _{-p,n}(K,\G^\circ L)
        \\
        &\geq \frac{n}{n+p}\vol[nm](L)\vol[n](K)^\frac{n+p}{n}\vol[n](\G^\circ L)^{-\frac{p}{n}},
    \end{align*}
    with equality if and only if $K$ is a homothet of $\G^\circ L$, up to sets of measure zero. The claim then follows from Lemma~\ref{l:hipBS} and \eqref{eq:elipp_cal}.
\end{proof}
We are now in a position to prove Theorem~\ref{t:BS_on_sphere}.

\begin{proof}[Proof of Theorem~\ref{t:BS_on_sphere}]
    Observe that, for $K\in\sta[n]_{n+p}$ and a compact set with non-empty interior $L\subset  M_{n,m}(\R)$ such that either $\G L$ or $\G^\circ L$ has center of mass at the origin, one has
    \begin{align*}
    &\int_{\s}\int_{\mathbb{S}^{nm-1}}h_Q(U^tv)^p\rho_K(v)^{n+p}\rho_L(U)^{nm+p}dUdv
    \\
    &=(nm+p)(n+p) \int_{\s}\int_{\mathbb{S}^{nm-1}}\left(h_Q(U^tv)^p\int_{0}^{\rho_K(v)}r^{n+p-1}dr\int_{0}^{\rho_L(U)}s^{nm+p-1}ds\right)dUdv
    \\
    &=(nm+p)(n+p)\int_K\int_Lh_Q(A^tu)^p dA du.
    \end{align*}
Therefore, applying Lemma~\ref{l:LZBS}, we have
\begin{equation}
\begin{split}
    &\int_{\s}\int_{\mathbb{S}^{nm-1}}h_Q(U^tv)^p\rho_K(v)^{n+p}\rho_L(U)^{nm+p}dUdv
    \\
    &\geq nm\frac{\vol[nm](L)\vol[n](K)}{\vol(\B)}\left(\frac{\vol[nm](L)}{\Vol(\PP \B )}\left(\frac{\vol[n](K)}{\vol[n](\B)}\right)^m\right)^\frac{p}{nm} .
    \end{split}
    \label{eq:almost_functional_matrices}
\end{equation}
    For given $f\in L^1(\mathbb{S}^{nm-1})$ and $g\in L^1(\s),$ define $K\in \mathcal{S}^n_{n+p}$ and $L\in \mathcal{S}^{nm}_{nm+p}$ via $\rho_K=g^\frac{1}{n+p}$ and $\rho_L=f^\frac{1}{nm+p}$. Note that $\G L$ and $\G^\circ L$ are origin-symmetric. Then, observe that
    \begin{align*}
        \|f\|_{L^\frac{nm}{nm+p}\left(\mathbb{S}^{nm-1}\right)} &= \left(\int_{\mathbb{S}^{nm-1}}f^\frac{nm}{nm+p}(U)dU\right)^\frac{nm+p}{nm}
        \\
        & = (nm)^\frac{nm+p}{nm}\left(\frac{1}{nm}\int_{\mathbb{S}^{nm-1}}\rho_L^{nm}(U)dU\right)^\frac{nm+p}{nm}
        \\
        & = (nm\vol[nm](L))^\frac{nm+p}{nm}.
    \end{align*}
    Similarly, we have that
    \begin{align*}
        \|g\|_{L^\frac{n}{n+p}\left(\s\right)} &= \left(\int_{\s}g^\frac{n}{n+p}(v)dv\right)^\frac{n+p}{n}
        = n^\frac{n+p}{n}\left(\frac{1}{n}\int_{\s}\rho_K^{n}(v)dv\right)^\frac{n+p}{n}
        \\
        & = (n\vol[n](K))^\frac{n+p}{n}.
    \end{align*}
    Inserting these computations into \eqref{eq:almost_functional_matrices} completes the proof of the inequality. As for the equality conditions, we must have that $K$ is a dilate of $\G^\circ L$ and $L=\PP E$ for some centered ellipsoid $E=A\B$ (up to null sets). This means from \eqref{eq:elipp_cal} that $K$ is a dilate of $(A\B)^{\circ}=A^{-t}\B$. Recalling that $g=\rho_K^{n+p},$ we have $g(u)=c_1\|u\|_{A^{-t}B_2^n}^{-(n+p)}=c_1|A^tu|^{-(n+p)}$. For $L$, write $L=\PP (A\B)$. Then, it was shown \cite[Proposition 3.6]{HLPRY23_2} that $\PP (AK) = |\det(A)|^{-\frac{1}{p}}A\PP K.$ Recalling that $f=\rho_L^{nm+p}$, we have almost everywhere that
    \begin{align*}f(U)&= c_2\|U\|_{A\PP \B}^{-(nm+p)} =c_2\|A^{-1}U\|_{\PP \B}^{-(nm+p)}.
    \end{align*}
\end{proof}
As advertised in the introduction, we will use Theorem~\ref{t:BS_on_sphere} to establish Theorem~\ref{t:main}.
We define the following constant: set 
\[
D_{n,p} = \left(\frac{1}{n}\Gamma\left(1+\frac{n}{p}\right)\right)^{-\frac{p}{n}}
\]
and
\begin{equation}
\label{eq_my_sharp}
    D_{n, p, \lambda}\!=\!
    \begin{cases}
    \left(\frac{n(1-\lambda)}{(n+p) \lambda-n}\right)\left(\left(\frac{p\lambda}{(n+p) \lambda-n}\right)^\frac{1}{1-\lambda} \frac{\Gamma\left(\frac{1}{1-\lambda}-\frac{n}{p}\right)}{\Gamma\left(\frac{1}{1-\lambda}\right)}\!\right)^{-\frac{p}{n}}D_{n,p}, &\lambda \in (\frac{n}{n+p},1);
    \\
    \frac{n}{pe}D_{n,p}, & \lambda=1;
    \\
    \left(\frac{n(\lambda-1)}{p \lambda+n(\lambda-1)}\right) \left(\left(\frac{p \lambda}{p \lambda+n(\lambda-1)}\right)^{\frac{1}{1-\lambda}} \frac{\Gamma\left(\frac{\lambda}{\lambda-1}\right)}{\Gamma\left(\frac{\lambda}{\lambda-1}+\frac{n}{p}\right)}\right)^{-\frac{p}{n}}D_{n,p}, & \lambda>1.
    \end{cases}
\end{equation}
We state now an important proposition that relates the volume of the Ball body of a function and a random variable distributed with respect to the probability measure whose density is proportional to said function. It was stated implicitly in the proof of the main theorem in \cite{NVH19} when $\rho(u)=|A^{-t}u|^{-1}$ for some $A\in GL_n(\R)$. The proof in the general situation is roughly the same; we provide a proof, following the ideas of Nguyen, for coherence and completeness.

\begin{proposition}
\label{p:ball_random_relate}
    Fix $p\geq 1$ and $\lambda > \frac{n}{n+p}$. Let $\mathcal{Y}$ be a random vector on $\R^n$ with finite moment of order p and density $g$. Then,
    \begin{equation}
\label{eq:ball_lambda_inequality}
    (n\vol[n](K_{g}(n+p)))^\frac{n+p}{n} \geq D_{n, p, \lambda}N_\lambda (\mathcal{Y})^p.
\end{equation}
There is equality if and only if there exists a function $\rho$ on $\s$, that is $(-1)$-homogeneously extended to $\R^n$, such that for a.e. $y\in\R^n$,
\begin{equation}
g(y)=\beta p_{p,\lambda}\left(\frac{\alpha}{\rho(y)}\right)
        \label{eq:general_vector_density}
    \end{equation}
    where in each case, $\alpha,\beta>0$ are nonnegative constants, chosen in such way so that $g$ integrates to $1$. Furthermore, $\rho_{K_g(n+p)}=c\rho$ almost everywhere for some $c>0$.
\end{proposition}

We first need the following lemma, the so-called Carlson-Levin inequality (see \cite[Lemma 4.1]{LYZ04_3} and \cite[Lemma 1.2]{NVH19}). This we provide without proof.
\begin{lemma}
    \label{l:carlson_levin}
    Let $p>0$ and $\lambda >\frac{n}{n+p}$. Let $f\in L^1(\R^n)$ have finite $p$th moment. Then, it holds:
    \begin{enumerate}
        \item If $\frac{n}{n+p}<\lambda <1,$ then
    \[
    \left(\frac{D_{n,p,\lambda}^\frac{n}{p}}{n\omega_n}\right)^\frac{1-\lambda}{\lambda} \left(\int_{\R^n}f(x)^\lambda dx \right)^\frac{1}{\lambda}\leq \left(\int_{\R^n}f(x)dx\right)^{1-\frac{n(1-\lambda)}{p\lambda}}\left(\int_{\R^n}|x|^p f(x)dx\right)^\frac{n(1-\lambda)}{p\lambda},
    \]
    and equality holds if and only if 
    \[
    f(x)=a(1+|bx|^p)^{-\frac{1}{1-\lambda}}
    \]
    for some $a,b>0$.
    
    \item If $\lambda >1$, then
    \begin{align*}\left(\frac{D_{n,p,\lambda}^\frac{n}{p}}{n\omega_n}\right)^\frac{p(\lambda-1)}{(n+p)\lambda-n}&\left(\int_{\R^n}f(x)dx\right) 
    \\
    &\leq \left(\int_{\R^n} f(x)^\lambda dx\right)^{\left(\frac{p}{n(\lambda-1)+p\lambda}\right)}\left(\int_{\R^n}|x|^pf(x)dx\right)^{\frac{n(\lambda -1)}{n(\lambda -1)+p\lambda}},\end{align*}
    and equality holds if and only if \[
    f(x) = a(1-|bx|^p)_+^{\frac{1}{\lambda-1}}
    \]
    for some $a,b>0$.

    \item If $\lambda=1$, then 
    \[
    \left(\int_{\R^n}|x|^pf(x)dx\right) \geq  \left(\frac{D_{n,p,1}}{(n\omega_n)^\frac{p}{n}}\right) \text{exp}\left(-\frac{p}{n}\frac{\int_{\R^n} f(x)\log f(x) dx}{\int_{\R^n}f(x)dx}\right)\left(\int_{\R^n}f(x)dx\right)^\frac{n+p}{n},
    \]
    with equality if and only if 
    \[
    f(x)=ae^{-|bx|^p}
    \]
    for some $a,b>0$.
    \end{enumerate}
\end{lemma}

We will need the usual Beta function $B(x,y)$ which we recall as we need a particular formulation of it, i.e. for $x,y>0$,
\begin{equation}
\label{eq:beta_function}
    B(x,y) = \frac{\Gamma(x)\Gamma(y)}{\Gamma(x+y)} = \int_0^1 t^{x-1}(1-t)^{y-1}dt = \int_0^\infty \frac{t^{x-1}}{(1+t)^{x+y}} dt.
\end{equation} 

Finally, we need the reverse H\"older's inequality. If $f$ and $g$ are nonnegative functions on the measure space $(\Omega,\mu)$ and real numbers $q$ and $r$ satisfy $q\in (0,1)$ and $\frac{1}{q}+\frac{1}{r}=1,$ then
\begin{equation}
\label{eq:holder_revserse}
    \int_{\Omega} f(x)g(x)d\mu(x) \geq \left(\int_{\Omega}f(x)^qd\mu(x)\right)^\frac{1}{q}\left(\int_{\Omega}g(x)^rd\mu(x)\right)^\frac{1}{r},
\end{equation}
with equality if and only if $g=c f^{q-1}$ $\mu$ a.e. on $\Omega$ for some $c\geq 0$.
\begin{proof}[Proof of Proposition~\ref{p:ball_random_relate}]
    We first consider the case where $\frac{n}{n+p} <\lambda < 1$. For a fixed $u\in \s$, we apply Lemma~\ref{l:carlson_levin} to the function $g_u=g(|x|u)$ and obtain
    \begin{equation}
    \label{eq:using_lemma_first}
    \left(\frac{D_{n,p,\lambda}^\frac{n}{p}}{n\omega_n}\right)^\frac{1-\lambda}{\lambda} \!\left(\int_{\R^n}g_u(x)^\lambda dx \right)^\frac{1}{\lambda}\!\leq\! \left(\int_{\R^n}g_u(x)dx\right)^{1-\frac{n(1-\lambda)}{p\lambda}}\!\left(\int_{\R^n}|x|^p g_u(x)dx\right)^\frac{n(1-\lambda)}{p\lambda}.
    \end{equation}
    By integrating in polar coordinates, one can readily verify that
    \begin{equation}
    \label{eq:ball_bodies_polar}
        \int_{\R^n}|x|^p g_u(x)dx = n\omega_n \rho_{K_g(n+p)}^{n+p}(u).
    \end{equation}
    Therefore, \eqref{eq:using_lemma_first} becomes
    \begin{equation}
    \label{eq:using_lemma_first_2}
        \left(\frac{D_{n,p,\lambda}^\frac{n}{n+p}}{n\omega_n}\right) \!\left(\!\int_{\R^n}g_u(x)^\lambda dx \!\right)^{\frac{p}{(n+p)(1-\lambda)}}\left(\!\int_{\R^n}g_u(x)dx\!\right)^{-\frac{(n+p)\lambda-n}{(n+p)(1-\lambda)}}\!\leq  \!\rho_{K_g(n+p)}^{n}(u).
    \end{equation}
    We then integrate both sides of \eqref{eq:using_lemma_first_2} over $\s$ and use \eqref{eq:holder_revserse} with $q=\frac{(n+p)(1-\lambda)}{p}\in (0,1)$ and $r=-\frac{(n+p)(1-\lambda)}{(n+p)\lambda-n} < 0$ to obtain
    \begin{equation}
    \label{eq:using_lemma_first_3}
    \begin{split}
        &\left(\frac{D_{n,p,\lambda}^\frac{n}{n+p}}{n\omega_n}\right) \left(\!\int_{\s}\!\int_{\R^n}\!g_u(x)^\lambda\! dx du\! \right)^{\frac{(n+p)(1-\lambda)}{p}}\!\left(\!\int_{\s}\!\int_{\R^n}\!g_u(x)dxdu\right)^{-\frac{(n+p)\lambda-n}{(n+p)(1-\lambda)}}
        \\
        &\leq \left(\frac{D_{n,p,\lambda}^\frac{n}{n+p}}{n\omega_n}\right) \int_{\s}\left(\int_{\R^n}g_u(x)^\lambda dx \right)^{\frac{p}{(n+p)(1-\lambda)}}\left(\int_{\R^n}g_u(x)dx\right)^{-\frac{(n+p)\lambda-n}{(n+p)(1-\lambda)}}du
        \\
        &\leq  (n\vol[n](K_g(n+p)).
        \end{split}
    \end{equation}
    From polar coordinates, we have
    \begin{equation}
    \label{eq:g_u_1}
    \int_{\s}\int_{\R^n}g_u(x)^\lambda dx du= \int_{\s}\int_{\s}\int_{0}^\infty g(|r\theta|u)^\lambda r^{n-1}drd\theta du = n\omega_n \int_{\R^n}g(x)^\lambda dx,
    \end{equation}
    and
    \begin{equation}
    \label{eq:g_u_2}
    \int_{\s}\int_{\R^n}g_u(x) dx du\!=\! \int_{\s}\int_{\s}\int_{0}^\infty g(|r\theta|u) r^{n-1}drd\theta du \!=\! n\omega_n \int_{\R^n}g(x) dx = n\omega_n.
    \end{equation}
    Inserting \eqref{eq:g_u_1} and \eqref{eq:g_u_2} into \eqref{eq:using_lemma_first_3} yields
    \begin{equation}
    \label{eq:using_lemma_first_4}
    \begin{split}
        D_{n,p,\lambda}^\frac{n}{n+p} \left(\int_{\R^n}g(x)^\lambda dx \right)^{\frac{p}{(n+p)(1-\lambda)}}
        \leq  (n\vol[n](K_g(n+p)).
        \end{split}
    \end{equation}
    Raising both sides of \eqref{eq:using_lemma_first_4} to the $\frac{n+p}{n}$ power and inserting the definition of $N_\lambda (\mathcal{Y})$ yields the claimed inequality.

    Suppose there is equality. Then, from our use of Lemma~\ref{l:carlson_levin}, we must have that, if we fix $u\in \s$, then, for a.e. $x\in\R^n$, it holds \begin{equation}
\label{eq:ball_bodies_radial_eq}
g(|x|u)=a(u)\left(1+|b(u) x|^p\right)^{-\frac{1}{1-\lambda}},\end{equation}
for some nonnegative functions $a$ and $b$. Also, from equality in the reverse H\"older's inequality \eqref{eq:holder_revserse}, we have for a.e. $u\in \s$ the existence of a positive constant $\gamma$ such that
\begin{equation}
\label{eq:ball_eq_hoelder}
\int_{\R^n}\! g(|x|u)^\lambda dx=\gamma \int_{\R^n} g(|x|u) dx.
\end{equation}
Inserting \eqref{eq:ball_bodies_radial_eq} into \eqref{eq:ball_eq_hoelder}, we obtain using \eqref{eq:beta_function} that for almost all $u\in\s$ it holds
\[
a(u) := \alpha=\left(\frac{\lambda p -n(1-\lambda)}{\lambda p}\gamma\right)^\frac{1}{\lambda -1}.
\]
Let $\rho$ be any constant multiple of $\rho_{K_g(n+p)}$; without loss of generality, \begin{equation}
\rho \equiv \rho_{K_g(n+p)}.
\label{eq:psi_def}
\end{equation} Then, from \eqref{eq:ball_bodies_radial_eq} and \eqref{eq:ball_bodies_polar}, we have for a.e. $u\in\s$
\begin{equation}
    \frac{p\lambda-n(1-\lambda)}{n(1-\lambda)}\rho(u)^{n+p} = \frac{\alpha b(u)^{-(n+p)}}{p}B\left(\frac{n}{p},\frac{1}{1-\lambda}-\frac{n}{p}\right),
\end{equation}
by using \eqref{eq:beta_function}. These two equations together imply
\begin{align*}
 b(u) &= \left(\left(\frac{\lambda p -n(1-\lambda)}{\lambda p}\gamma\right)^\frac{1}{\lambda -1}B\left(\frac{n}{p},\frac{1}{1-\lambda}-\frac{n}{p}\right)\frac{n}{p}\frac{(1-\lambda)}{p\lambda-n(1-\lambda)}\right)^\frac{1}{n+p}\rho(u)^{-1}
 \\
 & = \left(\frac{\beta (1-\lambda)}{p}\right)^\frac{1}{p}\rho(u)^{-1},
\end{align*}
where $\beta$ is defined implicitly so equality holds. Putting the pieces together, we have it holds for a.e. $u\in\s$
\[
g(|x|u)=\alpha\left(1+\frac{\beta (1-\lambda)}{p}\left(\frac{|x|}{\rho(u)}\right)^p\right)^{-\frac{1}{1-\lambda}}.
\]
By identifying $u=x/|x|$ and using the $(-1)$-homogeneity of $\rho$, the claimed formula for equality follows in this case.

The case when $\lambda >1$, for both the inequality and the equality condition, is essentially the same. Simply note that in the use of the reverse H\"older's inequality \eqref{eq:holder_revserse} one takes $q=\frac{(n+p)(\lambda-1)}{(n+p)\lambda-n} \in (0,1)$ and $r=-\frac{(n+p)(\lambda-1)}{p}<0$.

Finally, for $\lambda =1,$ we again apply Lemma~\ref{l:carlson_levin} to $g_u(x)=g(|x|u)$ and obtain
    \[\rho_{K_g(n+p)}^{n}(u) \geq  \left(\frac{D_{n,p,1}^\frac{n}{n+p}}{n\omega_n}\right)\text{exp}\left(-\frac{p}{n+p}\frac{\int_{\R^n} g_u(x)\log g_u(x) dx}{\int_{\R^n}g_u(x)dx}\right)\left(\int_{\R^n}g_u(x)dx\right),
    \]
    where we have already applied \eqref{eq:ball_bodies_polar}. Integrating both sides over $\s$ yields
    \begin{align*}&(n\vol[n](K_g(n+p)) 
    \\
    &\geq \left(\frac{D_{n,p,1}^\frac{n}{n+p}}{n\omega_n}\right)\int_{\s} \text{exp}\left(-\frac{p}{n+p}\frac{\int_{\R^n} g_u(x)\log g_u(x) dx}{\int_{\R^n}g_u(x)dx}\right)\left(\int_{\R^n}g_u(x)dx\right)du.
    \end{align*}
    Using Jensen's inequality for the convex function $t\mapsto e^{-\frac{p}{n+p}t}$, we obtain
    \begin{align*}&(n\vol[n](K_g(n+p)) 
    \\
    &\geq D_{n,p,1}^\frac{n}{n+p}\int_{\s} \text{exp}\left(-\frac{p}{n+p}\frac{\int_{\R^n} g_u(x)\log g_u(x) dx}{\int_{\R^n}g_u(x)dx}\right)\left(\int_{\R^n}g_u(x)dx\right)\frac{du}{n\omega_n}
    \\
    &\geq 
    D_{n,p,1}^\frac{n}{n+p} \text{exp}\left(-\frac{p}{n+p}\int_{\s}\int_{\R^n} g_u(x)\log g_u(x) dx\frac{du}{n\omega_n}\right)
    \\
    & =
    D_{n,p,1}^\frac{n}{n+p} \text{exp}\left(-\frac{p}{n+p}\int_{\R^n} g(x)\log g(x) dx\right).
    \end{align*}
Inserting the definition of $N_1(\mathcal{Y})$ and raising both sides to the $\frac{n+p}{n}$ power yields the inequality.

For the equality conditions, equality in our use of Lemma~\ref{l:carlson_levin} yields if we fix $u\in \s$, then for a.e. $x\in \R^n$ we have: 
\begin{equation}
\label{eq:ball_bodies_radial_eq_2}
g(|x|u)=a(u)e^{-|b(u)x|^p},\end{equation}
and, from equality in the use of Jensen's inequality, we have for a.e. $u\in \s$ the existence of a positive constant $\gamma$ such that
\begin{equation}
\label{eq:ball_eq_jensen}
\int_{\R^n} g(|x|u) \log g(|x|u)  dx=\gamma \int_{\R^n}  g(|x|u) dx.
\end{equation}
By inserting \eqref{eq:ball_bodies_radial_eq_2} into \eqref{eq:ball_eq_jensen}, we obtain using $\rho$ from \eqref{eq:psi_def} a formula for $a(u)$ a.e.:
\[
a(u) = e^{\gamma}e^{\int_{\R^n} |x|^p\frac{e^{-|x|^p}dx}{\int_{\R^n}e^{-|x|^p}dx}},
\]
which is just a constant, say $\alpha$.
Then, inserting \eqref{eq:ball_bodies_radial_eq_2} and the above formula for $a(u)$ into \eqref{eq:ball_bodies_polar}, we obtain that
\begin{align*}
b(u) = \left(\Gamma\left(1+\frac{n}{p}\right) \frac{\alpha}{p} \right)^\frac{1}{n+p}\rho(u)^{-1}=\left(\frac{\beta}{p}\right)^\frac{1}{p}\rho(u)^{-1}, \,\text{where}\quad \beta = p\left(\Gamma\left(1+\frac{n}{p}\right) \frac{\alpha}{p} \right)^\frac{p}{n+p}.
\end{align*}
We then have that
\[
g(|x|u)=\alpha e^{-\frac{\beta}{p}\left(\frac{|x|}{\rho(u)}\right)^p}.
\]
Again identifying $u=x/|x|$ and using the $(-1)$-homogeneity of $\rho$ yields the claimed formula for the equality case.
\end{proof}

With the necessary preparation completed, we can now prove Theorem~\ref{t:main}.

\begin{proof}[Proof of Theorem~\ref{t:main}]
Let $\mathfrak{X}$ and $\mathcal{Y}$ be two random vectors, one on $M_{n,m}(\R)$ and one on $\R^n$, respectively, each with finite moment of order $p$ and densities $f$ and $g$, respectively. Consider the $(nm+p)$th Ball body of $f$ and the $(n+p)$th Ball body of $g$. Then, from \eqref{eq:ball_body}, one has that
\[
\int_{\mathbb{S}^{nm-1}}\rho_{K_{f}(nm+p)}(U)^{nm+p}dU = \int_{M_{n,m}(\R)}|A|^pf(A)dA < \infty
\]
and
\[
\int_{\s}\rho_{K_{g}(n+p)}(u)^{n+p}du = \int_{\R^n}|x|^pg(x)dx < \infty.
\]

\noindent Additionally, one has from polar coordinates that $\mathbb{E}[h_Q(\mathfrak{X}^t\mathcal{Y})^p]$ given by \eqref{eq:expectation} satisfies the inequality
\begin{align*}
    \mathbb{E}[h_Q(\mathfrak{X}^t\mathcal{Y})^p]&= \int_{\s}\int_{\mathbb{S}^{nm-1}}h_Q(U^tu)^p \rho_{K_{f}(nm+p)}(U)^{nm+p} \rho_{K_{g}(n+p)}(u)^{n+p} dU du
    \\
    &\geq \frac{\|\rho_{K_{f}(nm+p)}^{nm+p}\|_{L^\frac{nm}{nm+p}\left(\mathbb{S}^{nm-1}\right)} \|\rho_{K_{g}(n+p)}^{n+p}\|_{L^\frac{n}{n+p}\left(\s\right)}} {(n\vol(\B))^{\frac{n+p}{n}}\left(nm\Vol(\PP \B )\right)^{\frac{p}{nm}}},
\end{align*} 
where we used Theorem~\ref{t:BS_on_sphere} in the last line. Next, observe that
\begin{align*}
    \|\rho_{K_{f}(nm+p)}^{nm+p}\|_{L^\frac{nm}{nm+p}\left(\mathbb{S}^{nm-1}\right)} &= \left(\int_{\mathbb{S}^{nm-1}} \rho_{K_{f}(nm+p)}(U)^{nm}dU\right)^\frac{nm+p}{nm}
    \\
    &=(nm\vol[nm](K_{f}(nm+p)))^\frac{nm+p}{nm},
\end{align*}
and, similarly,
\begin{align*}
    \|\rho_{K_{g}(n+p)}^{n+p}\|_{L^\frac{n}{n+p}\left(\s\right)} &= \left(\int_{\s} \rho_{K_{g}(n+p)}(u)^{n}du\right)^\frac{n+p}{n}
    =(n\vol[n](K_{g}(n+p)))^\frac{n+p}{n}.
\end{align*}
Inserting this, we have
\begin{align*}
    \mathbb{E}[h_Q(\mathfrak{X}^t\mathcal{Y})^p] &\geq  {(n\vol(\B))^{-\frac{n+p}{n}}\left(nm\Vol(\PP \B )\right)^{-\frac{p}{nm}}}
    \\
    &\quad\quad\times(nm\vol[nm](K_{f}(nm+p)))^\frac{nm+p}{nm}(n\vol[n](K_{g}(n+p)))^\frac{n+p}{n}.
\end{align*} 
We then complete the proof of the inequality by invoking Proposition~\ref{p:ball_random_relate}. We note that, when applying the proposition to the function $f$ on $M_{n,m}(\R)$, we must take $\lambda >\frac{nm}{nm+p}$.

As for the equality conditions, we must have equality in our use of Theorem~\ref{t:BS_on_sphere}. Thus, we have 
\begin{equation}\label{eq:ball_bodies_eq}
\rho_{K_{g}(n+p)}(u) = c_1|A^{t}u|^{-1} \text{ and } \rho_{K_{f}(nm+p)}(U) = c_2\|A^{-1}U\|_{\PP \B}^{-1}.\end{equation}
for some $c_1,c_2>0$ and $A\in GL_n(\R)$. Then, the claim follows from the equality conditions of Proposition~\ref{p:ball_random_relate}.
\end{proof}

It was shown \cite{HLPRY23_2} that $\Pi^\circ_{Q,\infty} K$ exists and that as $p\to\infty,$ $\PP K \to \Pi^\circ_{Q,\infty} K$ continuously in any meaningful sense of the word (e.g. in the Hausdorff metric, or the point-wise convergence of their gauges). Recall the definition of $\lambda$-R\'enyi from \eqref{eq:renyi}. If $L\subset M_{n,m}(\R)$ and $K\subset \R^n$ are compact sets, $\mathfrak{X}$ is the random variable with density $\chi_L/\vol[nm](L)$, and $\mathcal{Y}$ is the random variable with density $\chi_K/\vol[n](K)$, then $N_\infty(\mathfrak{X}) = \vol[nm](L)^\frac{1}{nm}$ and $N_{\infty}(\mathcal{Y})=\vol[n](K)^\frac{1}{n}$. One can verify that
$$\lim_{p\to\infty}\lim_{\lambda\to\infty} D_{n,p,\lambda}^\frac{n}{p} = \lim_{p\to\infty} D_{n,p}^{\frac{n}{p}} = n.$$

Therefore, by sending $p\to\infty$ and $\lambda\to\infty$ in Theorem~\ref{t:main}, we obtain the following corollary, which extends on the $m=1$ case when $Q$ is a symmetric interval by Lutwak, Yang and Zhang \cite{LYZ04_3}.

\begin{corollary}
Let $L\subset M_{n,m}(\R)$ and $K\subset \R^n$ be compact sets and $Q\in \conbodo[m]$ be such that either $L$ or $Q$ is origin-symmetric. Then, 
\[
(\vol[n](\B))^{m}(\vol[nm](\Pi^\circ_{Q,\infty}\B))\max_{A\in L, v\in K}h_Q(A^tv)^{nm} \geq \vol[nm](L)\vol[n](K)^m.
\]
\end{corollary}

\section{Other related inequalities}
\label{sec:other_results}
While not strictly necessary, we take a moment to show that $\vol[n](K_g(n+p))$ is finite when the random variable with density $g$ has finite moment of order $p$. Indeed, this follows from Jensen's inequality:
\begin{equation}
\begin{split}
\label{eq:jensen_ball}
    (n\vol[n](K_g(n+p)))^\frac{n+p}{n} &= (n\omega_n)^\frac{n+p}{n}\left(\int_{\s}\rho_{K_g(n+p)}(u)^n \frac{du}{n\omega_n}\right)^\frac{n+p}{n} 
    \\
    &\leq (n\omega_n)^\frac{n+p}{n}\int_{\s}\rho_{K_g(n+p)}(u)^{n+p} \frac{du}{n\omega_n}
    \\
    & = (n\omega_n)^\frac{p}{n}\int_{\R^n}|y|^p g(y) dy.\
\end{split}
\end{equation}
Note that there is equality if and only if $\rho_{K_g(n+p)}$ is constant on $\s$ almost everywhere, i.e. $K_g(n+p)$ is a Euclidean ball, up to sets of measure zero. This occurs when $g(y)=q(|y|)$ for some nonnegative function $q$ on $\R$. The random variable $\mathcal{Y}$ with density $g$ is said to be \textit{spherically contoured.} The expected $p$th moment of a random variable $\mathcal{Y}$ with density $g$ is precisely
\[
\mathbb{E}(|\mathcal{Y}|^p) = \int_{\R^n}|y|^pg(y)dy,
\]
and so combining \eqref{eq:jensen_ball} with Proposition~\ref{p:ball_random_relate} yields the following inequality: fix $p\geq 1$ and $\lambda > \frac{n}{n+p}$. Let $\mathcal{Z}$ be the standard generalized Gaussian random vector associated with $(p,\lambda)$. Then, if $\mathcal{Y}$ is a random vector with finite moment of order $p$, one has
\begin{equation}
\label{eq:iso_vec}
    \frac{\mathbb{E}(|\mathcal{Y}|^p)}{\mathbb{E}(|\mathcal{Z}|^p)} \geq \left(\frac{N_\lambda (\mathcal{Y})}{N_\lambda (\mathcal{Z})}\right)^p,
\end{equation}
with equality when $\mathcal{Y}=t\mathcal{Z}$. This had been established previously in \cite{LYZ04_3,LYZ05_2,LYZ07}. 

We conclude by establishing some other moment-entropy-type inequalities. As we will see, these inequalities do not occur in the space of matrices, but, instead, are induced on $\R^n$ by the structure on $M_{n,m}(\R)$ and an origin-symmetric convex body $Q\in\conbodo[m]$. We must first recall the previously established situations for when $Q=[-1,1]$. In \cite{LLYZ13}, an affine invariant version of \eqref{eq:iso_vec} was established. If $p>0$ and $\mathcal{Y}$ is a random vector on $\R^n$ with finite $p$th moment, its affine $p$th moment is precisely
$$M_p(\mathcal{Y}) = (n\omega_n)^{\frac{n+p}{n}}\left(\left(\frac{2\omega_{n-2+p}}{\omega_{p-1}}\right)^{\frac{n}{p}} \int_{\s}\mathbb{E}_\mathcal{Y}(|\mathcal{Y}\cdot u|^p)^{-\frac{n}{p}}\frac{du}{n\omega_n}\right)^{-\frac{p}{n}}.$$
It then holds from H\"older's inequality that
$M_p(\mathcal{Y}) \leq \mathbb{E}(|\mathcal{Y}|^p)$ with equality when $\mathcal{Y}$ is spherically contoured and it was also shown that
\begin{equation}
\label{eq:affine_iso_vec}
    \frac{M_p(\mathcal{Y})}{N_\lambda (\mathcal{Y})^p}\geq D_{n,p,\lambda},
\end{equation}
with equality obtained only when $\mathcal{Y}$ is a generalized Gaussian random vector.

Let $Q\in\conbodo[m]$ and $p>0$. For a random vector $\mathcal{Y}$ on $\R^n$, we define its affine $(L^p,Q)$ moment as
\begin{equation}
    M_{Q,p}(\mathcal{Y}) =(n\omega_n)^{\frac{n+p}{n}} \left(\frac{\omega_{nm}}{\Vol(\!\PP \B\! )}\int_{\mathbb{S}^{nm-1}}\mathbb{E}_{\mathcal{Y}}(h_Q(U^t\mathcal{Y})^p)^{-\frac{nm}{p}}\frac{dU}{nm\omega_{nm}}\right)^{-\frac{p}{nm}}.
\end{equation}
\begin{theorem}
    Fix $m,n\in\N$ and $p\geq 1$. Suppose $Q\in\conbodio[m]$ is origin-symmetric and $\lambda \geq \frac{nm}{nm+p}$. Let $\mathcal{Y}$ be a random vector with finite $p$th moment. Then, it holds
\[
\frac{M_{Q,p}(\mathcal{Y})}{N_\lambda (\mathcal{Y})^p} \geq D_{n,p,\lambda} .
\]
\end{theorem}
\begin{proof}
In \cite[Lemma 4.1]{LYZ04_3}, the following extension of the dual Minkowski's first inequality \eqref{dual_Min_first} to random vectors was established; as a special case one obtains, if $p>0$, $\lambda > \frac{nm}{nm+p}$ and $\mathfrak{X}$ and $\mathcal{Y}$ are random vectors on $M_{n,m}(\R)$ and $\R^n$ respectively with finite $p$th moment, then
\begin{equation}
\label{eq:min_first_random_vec}
    \int_{M_{n,m}(\R)}\!\!\!\!\!\!\!\mathbb{E}_{\mathcal{Y}}(h_Q(\mathfrak{X}^t\mathcal{Y})^p)f_{\mathfrak{X}}(\mathfrak{X})d\mathfrak{X} \!\geq\! N_\lambda(\mathfrak{X})^pD_{nm,p,\lambda} \!\left(\!\int_{\mathbb{S}^{nm-1}}\!\!\!\mathbb{E}_{\mathcal{Y}}(h_Q(U^t\mathcal{Y})^p)^{-\frac{nm}{p}}dU\!\right)^{-\frac{p}{nm}},
\end{equation}
with equality if and only if there exists $a,b>0$ such that \begin{equation}
\label{eq:psi_for_y}
f_{\mathfrak{X}}(A)=bp_{p,\lambda}\left(a\mathbb{E}_{\mathcal{Y}}(h_Q(A^t\mathcal{Y})^p)^{\frac{1}{p}}\right).\end{equation} Define
\[
c=D_{nm,p,\lambda}^{-\frac{1}{p}}\left((n\omega_n)^{\frac{n+p}{p}}\left(nm\Vol(\!\PP \B\! )\right)^\frac{1}{m}\right)^\frac{1}{n}.
\]
By performing Fubini's theorem on \eqref{eq:min_first_random_vec}, we obtain 
\begin{equation}
\label{eq:min_first_random_vec_2}
    \int_{\R^n}\mathbb{E}_{\mathfrak{X}}(h_Q(\mathfrak{X}^t\mathcal{Y})^p)g(\mathcal{Y})d\mathcal{Y} \geq M_{Q,p}(\mathcal{Y})
\end{equation}
whenever $N_\lambda (\mathfrak{X})=c$. Still using the choice of $f$ from \eqref{eq:psi_for_y} and normalizing $\mathfrak{X}$ so that $N_\lambda (\mathfrak{X})=c$ holds, we have equality in \eqref{eq:min_first_random_vec_2}. But the left-hand side of \eqref{eq:min_first_random_vec_2} is precisely $\mathbb{E}(h_Q(\mathfrak{X}^t\mathcal{Y})^p)$. Therefore, we obtain the claim from Theorem~\ref{t:main}.
\end{proof}

Another important inequality is the Fisher information inequality: for a random vector $\mathcal{Y}$ on $\R^n$ with finite first moment, it holds
$$N_1(\mathcal{Y})\Phi(\mathcal{Y}) \geq 2\pi n e,$$
with equality if and only if $\mathcal{Y}$ is a standard Gaussian random vector. Here, $\Phi(\mathcal{Y})$ is the Fisher information of $\mathcal{Y}$, which is given by 
\begin{equation}
    \Phi(\mathcal{Y}) = \int_{\R^n}g_{\mathcal{Y}}^{-1}(y)|\nabla g_{\mathcal{Y}}(y)|dy,
    \label{eq:fisher}
\end{equation}
where $g_{\mathcal{Y}}$ is the density of $\mathcal{Y}$. In \cite{LLYZ12}, Lutwak, Lv, Yang and Zhang generalized the Fisher information inequality to the parameters $(p,\lambda)$. For a random vector $\mathcal{Y}$ with density $g_{\mathcal{Y}}$, its $\lambda$-score is given by $$\mathcal{Y}_\lambda = g_{\mathcal{Y}}^{\lambda -2}(\mathcal{Y})\nabla g_{\mathcal{Y}}(\mathcal{Y}),$$ and its $(p,\lambda)$-Fisher information is given by
$$\Phi_{p,\lambda}(\mathcal{Y}) = \mathbb{E}_{\mathcal{Y}}(|\mathcal{Y}_\lambda|^p).$$ The classical case is when $p=2,\lambda =1$. The quantity $\Phi_{p,\lambda}(\mathcal{Y})$ is not affine invariant. They also defined the affine $(p,\lambda)$-Fisher information as
$$\Psi_{p,\lambda}(\mathcal{Y}) = M_p(\mathcal{Y}_\lambda).$$ The main result of Lutwak, Lv, Yang and Zhang is the following affine $(p,\lambda)$-Fisher information inequality \cite[Theorem 8.2]{LLYZ12}: suppose $\mathcal{Y}$ is a random vector, $1\leq p <n$ and $\lambda \geq \frac{n-1}{n}$. Then, it holds
\begin{equation}
\label{eq:affine_p_lambda_fisher}
    \Psi_{p,\lambda}(\mathcal{Y}) N_\lambda (\mathcal{Y})^{p((\lambda-1) n+1)} \geq \Phi_{p, \lambda}(\mathcal{Z}) N_\lambda (\mathcal{Z})^{p((\lambda-1) n+1)},
\end{equation}
with equality if and only if $\mathcal{Y}$ is a generalized Gaussian associated with $(p,\lambda)$. Since $\Psi_{p,\lambda}(\mathcal{Y}) \leq \Phi_{p,\lambda}(\mathcal{Y})$, this implies the inequality
\begin{equation}
\label{eq:p_lambda_fisher}
    \Phi_{p, \lambda}(\mathcal{Y}) N_\lambda (\mathcal{Y})^{p((\lambda-1) n+1)} \geq \Phi_{p, \lambda}(\mathcal{Z}) N_\lambda (\mathcal{Z})^{p((\lambda-1) n+1)},
\end{equation}
with equality if and only if $\mathcal{Y}$ is a standard generalized Gaussian associated with $(p,\lambda)$.

Following their path, we define for a random vector $\mathcal{Y}$ on $\R^n$ its affine $(Q,p,\lambda)$-Fisher information
$$\Psi_{Q,p,\lambda}(\mathcal{Y})=M_{Q,p}(\mathcal{Y}_\lambda).$$
Recall the definition of the $LYZ$ body of a weakly differentiable function $w$ from \eqref{eq:LYZw}. By defining the function $w=g_\mathcal{Y}^{\lambda-1+\frac{1}{p}}$, we deduce that
\begin{equation}
    \Psi_{Q,p,\lambda}(\mathcal{Y}) = \left(\lambda -1 + \frac{1}{p}\right)^{-p} \vol[nm]\left(\PP \langle w\rangle_p\right)^{-\frac{p}{nm}}.
\end{equation}
Following the idea of the proof of \cite[Theorem 8.2]{LLYZ12}, set $r=\frac{1}{\lambda-1+\frac{1}{p}}$, and define the parameters $q$ and $s$ via \eqref{eq:q_param} and \eqref{eq:s_param}. Then, one has the identities
\begin{equation}
    \|w\|_{L^q(\R^n)}^{1-s} = N_\lambda (\mathcal{Y})^{-((\lambda-1)n+1)} \, \text{and} \, \|w\|_{L^r(\R^n)}^{s}=1.
\end{equation}
The following theorem then follows from \eqref{eq:CNV_3}.
\begin{theorem}
    Fix $n,m\in\R$. Let $1\leq p <n$ and $\lambda \geq 1 - \frac{1}{n}$. Suppose $Q\in\conbodo[m]$ is origin-symmetric. Then,
    \begin{equation}
    \label{eq:CNV_4}
    \Psi_{Q,p,\lambda}(\mathcal{Y})  N_\lambda (\mathcal{Y})^{p((\lambda-1)n+1)} \geq \Vol(\PP \B )^{-\frac{p}{nm}} (n\omega_n)^{-1} \left(\frac{c(n,r,p)}{\lambda -1 + \frac{1}{p}}\right)^p,
\end{equation}
where $c(n,r,p)$ is the sharp constant from \eqref{eq:CNV_2}. There is equality if and only if $\mathcal{Y}$ is a generalized Gaussian.
\end{theorem}

\bibliographystyle{acm}
\bibliography{references}
\end{document}